\UseRawInputEncoding
\documentclass[11pt,twoside, leqno]{article}

\usepackage{mathrsfs}
\usepackage{amssymb}
\usepackage{amsmath}
\usepackage{amsthm}
\usepackage{amsfonts}

\usepackage[colorlinks=true, pdfstartview=FitV, linkcolor=blue, citecolor=blue, urlcolor=blue]{hyperref}
\usepackage{color}

\allowdisplaybreaks

\def\rr{{\mathbb R}}
\def\rn{{{\rr}^n}}
\def\zz{{\mathbb Z}}

\def\nn{{\mathbb N}}
\def\fz{\infty}

\def\ccc{{\mathbb C}}

\def\cs{{\mathcal S}}

\def\az{\alpha}

\renewcommand\tilde{\widetilde}

\def\supp{{\rm{\,supp\,}}}

\def\ls{\lesssim}
\def\lz{\lambda}

\def\ez{\varepsilon}

\def\sz{\sigma}

\def\hs{\hspace{0.3cm}}

\def\r{\right}
\def\lf{\left}

\def\bint{{\ifinner\rlap{\bf\kern.30em--}
\int\else\rlap{\bf\kern.35em--}\int\fi}\ignorespaces}

\def\sbint{{\ifinner\rlap{\bf\kern.32em--}
\hspace{0.078cm}\int\else\rlap{\bf\kern.45em--}\int\fi}\ignorespaces}

\def\dsup{\displaystyle\sup}

\newtheorem{theorem}{Theorem}[section]
\newtheorem{lemma}[theorem]{Lemma}

\theoremstyle{definition}
\newtheorem{remark}[theorem]{Remark}
\newtheorem{definition}[theorem]{Definition}
\numberwithin{equation}{section}

\numberwithin{equation}{section}

\numberwithin{equation}{section}

\begin{document}

\arraycolsep=1pt

\title{\Large\bf  The Characterizations of Anisotropic Mixed-Norm Hardy Spaces on $\mathbb{R}^n$ by Atoms and Molecules \footnotetext{\hspace{-0.35cm} {\it 2010 Mathematics Subject Classification}.
{Primary 42B20; Secondary 42B30, 46E30.}
\endgraf{\it Key words and phrases.} Anisotropy, Hardy space, molecule, atom.
}}
\author{Wenhua Wang and Aiting Wang}
\date{  }
\maketitle

\vspace{-0.8cm}

\begin{center}
\begin{minipage}{13cm}\small
{\noindent{\bf Abstract} \
Let $\vec{p}\in(0,\,\infty)^n$, $A$ be an expansive dilation on $\mathbb{R}^n$,
 and $H^{\vec{p}}_A({\mathbb {R}}^n)$ be the anisotropic mixed-norm Hardy space defined via the non-tangential grand
maximal function studied by \cite{hlyy20}. In this paper, the authors establish
new atomic and molecular decompositions of $H^{\vec{p}}_A({\mathbb {R}}^n)$.
As an application, the authors obtain
a boundedness criterion for a class of linear operators
from $H^{\vec{p}}_{A}(\mathbb{R}^n)$ to $H^{\vec{p}}_{A}(\mathbb{R}^n)$.
Part of results are still new even in the classical isotropic setting (in the case $A:=2\mathrm I_{n\times n}$, ${\mathrm{I}}_{n\times n}$ denotes the $n\times n$ unit matrix).}
\end{minipage}
\end{center}

\section{Introduction}
\hskip\parindent
The theory of Hardy spaces on the Euclidean space $\rn$ has been developed and plays an important role in various fields of analysis and partial differential equations; see, for examples, \cite{clr06,dz20,fs72,s60,t17,zy16}.
On the other hand,
the mixed-norm Lebesgue space $L^{\vec{p}}(\rn)$, with the exponent vector $\vec{p}\in (0, \,\infty]^n$, is a natural generalization of the classical Lebesgue space
$L^p(\rn)$, via replacing the constant exponent $p$ by an exponent vector $\vec{p}$. The study of mixed-norm Lebesgue spaces originates from Benedek and Panzone \cite{bp61}. Later on, function spaces in mixed-norm setting have attracted considerable
attention and have rapidly been developed; see, for instance,
\cite{cg20,cs20,cs21,hwyy21}.

Let vector $\vec{p}:=(p_1,\ldots,p_n)\in (0, \,\infty]^n$. Recently, Cleanthous et al. \cite{cgn17} introduced
the anisotropic mixed-norm Hardy space $H_{\vec{a}}^{\vec{p}}({\mathbb {R}}^n)$, where $\vec{a}:=(a_1,\ldots,a_n)\in[1,\,\infty)^n$, via the non-tangential grand maximal function, and then obtained its maximal function characterizations. Not long afterward, Huang et al. \cite{hlyy19} further completed some real-variable characterizations of the space, such as the characterizations in terms of the atomic characterization and the Littlewood-Paley function characterization. Moreover, they obtained the boundedness of $\delta$-type Calder\'on-Zygmund operators from $H^{\vec{p}}_{\vec{a}}(\rn)$ to $L^{\vec{p}}(\rn)$ or from $H^{\vec{p}}_{\vec{a}}(\rn)$ to itself. For more information about this space, see \cite{hcy21,hlyy19y,hy21}

Very recently, Huang et al. \cite{hlyy20}
also introduced the new anisotropic mixed-norm Hardy space $H^{\vec{p}}_A({\mathbb {R}}^n)$ associated with a general expansive matrix $A$, via the non-tangential grand maximal function, and then established its various real-variable characterizations of $H_A^{\vec{p}}$, respectively, in terms of the atomic characterization and  the Littlewood-Paley function characterization. For more information about Hardy space, see \cite{hlyy20,hlyy20x,s12,t19,wl12}.   Nevertheless, the molecular decompositions of $H^{\vec{p}}_A({\mathbb {R}}^n)$ has
not been established till now. Once its molecular decomposition is established,
it can be conveniently used to prove the boundedness of many important operators on the space
$H^{\vec{p}}_A({\mathbb {R}}^n)$, for example, one of the most famous operator in harmonic analysis, Calder\'on-Zygmund operators.
 To complete the theory of the new anisotropic mixed-norm Hardy space $H^{\vec{p}}_A({\mathbb {R}}^n)$, in this article, we establish the molecular decompositions of $H^{\vec{p}}_A({\mathbb {R}}^n)$. Then, as application, we further obtain
a boundedness criterion for a class of linear operators
from $H^{\vec{p}}_{A}(\mathbb{R}^n)$ to $H^{\vec{p}}_{A}(\mathbb{R}^n)$.

Precisely, this article is organized as follows.

In Section \ref{s2}, we first recall some notations and definitions
concerning expansive dilations, the mixed-norm Lebesgue space $L^{\vec{p}}(\rn)$ and the anisotropic mixed-norm Hardy space $H^{\vec{p}}_A({\mathbb {R}}^n)$, via the non-tangential grand maximal function.

The aim of Section \ref{s3} is  to establish a new atomic characterization of $H^{\vec{p}}_A({\mathbb {R}}^n)$.

In Section \ref{s4},
 motivated by Liu et al. \cite{lyy17x,lyy16} and Huang et al. \cite{hlyy20}, we introduce the anisotropic mixed-norm molecular Hardy space $H^{\vec{p},\,q,\,s,\,\varepsilon}_{A,\,\mathrm{mol}}(\rn)$
and establish its equivalence with $H^{\vec{p}}_{A}(\rn)$ in Theorem \ref{t2.6}.
When it comes back to the isotropic setting, i.e., $A:=2{\rm I}_{n\times n}$,
${\rm I}_{n\times n}$ denotes the $n\times n$ unit matrix,
this result is still new, see Remark \ref{r2.7} for more details.
 It is worth pointing
out that some of the proof methods of the molecular characterization of $H^{p}_{A}(\rn)=H^{p,\,p}_{A}(\rn)$ (\cite[Theorem 3.9]{lyy16}) don't work anymore in the present setting.
For example, we search out some estimates related to $L^{\vec{p}}(\rn)$ norms for some series of functions which can be reduced into dealing with the $L^q(\rn)$ norms of the corresponding functions (see Lemma \ref{l4.1x}). Then, by using this key lemma and the Fefferman-Stein vector-valued inequality of the Hardy-Littlewood maximal operator $M_{\mathrm{HL}}$ on $L^{\vec{p}}(\rn)$ (see Lemma \ref{l3.5}), we prove
 their equivalences with $H^{\vec{p}}_{A}(\rn)$ and $H^{\vec{p},\,q,\,s,\,\varepsilon}_{A,\,\mathrm{mol}}(\rn)$.

In Section \ref{s5}, as an application of the molecular characterization of $H^{\vec{p}}_{A}(\rn)$, we obtain a boundedness criterion for a class of linear operators
from $H^{\vec{p}}_{A}(\mathbb{R}^n)$ to $H^{\vec{p}}_{A}(\mathbb{R}^n)$ (see Theorem \ref{t4.2} below). Particularly, when it comes back to the isotropic setting, i.e., $A:=2{\rm I}_{n\times n}$,
${\rm I}_{n\times n}$ denotes the $n\times n$ unit matrix, this result is also new.

Finally, we make some conventions on notation.
Let $\nn:=\{1,\, 2,\,\ldots\}$ and $\zz_+:=\{0\}\cup\nn$.
For any $\az:=(\az_1,\ldots,\az_n)\in\zz_+^n:=(\zz_+)^n$, let
$|\az|:=\az_1+\cdots+\az_n$ and
$\partial^\az:=
\lf(\frac{\partial}{\partial x_1}\r)^{\az_1}\cdots
\lf(\frac{\partial}{\partial x_n}\r)^{\az_n}.$
Throughout the whole paper, we denote by $C$ a \emph{positive
constant} which is independent of the main parameters, but it may
vary from line to line.  For any $q\in[1,\,\infty]$, we denote by $q'$ its conjugate index, namely, $1/q + 1/{q'}=1$.
For any $a\in\rr$, $\lfloor a\rfloor$ denotes the
\emph{maximal integer} not larger than $a$.
The \emph{symbol} $D\ls F$ means that $D\le
CF$. If $D\ls F$ and $F\ls D$, we then write $D\sim F$.
If $E$ is a
subset of $\rn$, we denote by $\chi_E$ its \emph{characteristic
function}. If there are no special instructions, any space $\mathcal{X}(\rn)$ is denoted simply by $\mathcal{X}$. For instance, $L^2(\rn)$ is simply denoted by $L^2$. The symbol $C^{\infty}$ denotes the set of all {\it infinitely differentiable functions} on $\rn$. Denote by $\cs$   the \emph{space of all Schwartz functions} and $\cs'$
its \emph{dual space} (namely, the \emph{space of all tempered distributions}).



\section{Preliminary \label{s2}}
\hskip\parindent
In this section, we first recall the notion of anisotropic mixed-norm Hardy space $H^{\vec{p}}_{A}$, via the non-tangential grand maximal function $M_N(f)$, and then given its molecular decomposition.

We begin with recalling the notion of {{expansive dilations}}
on $\rn$; see \cite[p.\,5]{b03}. A real $n\times n$ matrix $A$ is called an {\it
expansive dilation}, shortly a {\it dilation}, if
$\min_{\lz\in\sz(A)}|\lz|>1$, where $\sz(A)$ denotes the set of
all {\it eigenvalues} of $A$. Let $\lz_-$ and $\lz_+$ be two {\it positive numbers} such that
$$1<\lz_-<\min\{|\lz|:\ \lz\in\sz(A)\}\le\max\{|\lz|:\
\lz\in\sz(A)\}<\lz_+.$$ Then there exists a positive constant $C$, such that, for any $x\in\rn$, when $j\in\zz_+$,
\begin{eqnarray}
&&C^{-1}(\lambda_-)^j|x|\leq|A^jx|\leq C(\lambda_+)^j|x|\label{e2.3v}
\end{eqnarray}
and, when $j\in\zz\backslash\zz_+$,
\begin{eqnarray}
&&C^{-1}(\lambda_+)^j|x|\leq|A^jx|\leq C(\lambda_-)^j|x|.\label{e2.4v}
\end{eqnarray}

From \cite[p.\,5, Lemma 2.2]{b03} that, for a fixed dilation $A$,
there exist a number $r\in(1,\,\fz)$ and a set $\Delta:=\{x\in\rn:\,|Px|<1\}$, where $P$ is some non-degenerate $n\times n$ matrix, such that $$\Delta\subset r\Delta\subset A\Delta,$$ and we may
assume that $|\Delta|=1$, where $|\Delta|$ denotes the
$n$-dimensional Lebesgue measure of the set $\Delta$. Let
$B_k:=A^k\Delta$ for $k\in \zz.$ Then $B_k$ is open, $B_k\subset
rB_k\subset B_{k+1}$ and $|B_k|=b^k$, here and hereafter, $b:=|\det A|$.
An ellipsoid $x+B_k$ for some $x\in\rn$ and $k\in\zz$ is called a {\it dilated ball}.
Denote by $\mathfrak{B}$ the set of all such dilated balls, namely,
\begin{eqnarray}\label{e2.1}
\mathfrak{B}:=\{x+B_k:\ x\in \rn,\,k\in\zz\}.
\end{eqnarray}
Throughout the whole paper, let $\sigma$ be the {\it smallest integer} such that $2B_0\subset A^\sigma B_0$
and, for any subset $E$ of $\rn$, let $E^\complement:=\rn\setminus E$. Then,
for all $k,\,j\in\zz$ with $k\le j$, it holds true that
\begin{eqnarray}
&&B_k+B_j\subset B_{j+\sz},\label{e2.3}\\
&&B_k+(B_{k+\sz})^\complement\subset
(B_k)^\complement,\label{e2.4}
\end{eqnarray}
where $E+F$ denotes the {\it algebraic sum} $\{x+y:\ x\in E,\,y\in F\}$
of  sets $E,\, F\subset \rn$.

\begin{definition}
 A \textit{quasi-norm}, associated with
dilation $A$, is a Borel measurable mapping
$\rho_{A}:\rr^{n}\to [0,\infty)$, for simplicity, denoted by
$\rho$, satisfying
\begin{enumerate}
\item[\rm{(i)}] $\rho(x)>0$ for all $x \in \rn\setminus\{ \mathbf{0}\}$,
here and hereafter, $\mathbf{0}$ denotes the origin of $\rn$;

\item[\rm{(ii)}] $\rho(Ax)= b\rho(x)$ for all $x\in \rr^{n}$, where, as above, $b:=|\det A|$;

\item[\rm{(iii)}] $ \rho(x+y)\le C_A\lf[\rho(x)+\rho(y)\r]$ for
all $x,\, y\in \rr^{n}$, where $C_A\in[1,\,\fz)$ is a constant independent of $x$ and $y$.
\end{enumerate}
\end{definition}

In the standard dyadic case $A:=2{{\rm I}_{n\times n}}$, $\rho(x):=|x|^n$ for
all $x\in\rn$ is
an example of homogeneous quasi-norms associated with $A$, here and hereafter, ${\rm I}_{n\times n}$ denotes the $n\times n$ {\it unit matrix},
$|\cdot|$ always denotes the {\it Euclidean norm} in $\rn$.

By \cite[Lemma 2.4]{b03}, we see
that all homogeneous quasi-norms associated with a given dilation
$A$ are equivalent. Therefore, for a
given dilation $A$, in what follows, for simplicity, we
always use the {\it{step homogeneous quasi-norm}} $\rho$ defined by setting,  for all $x\in\rn$,
\begin{equation*}
\rho(x):=\sum_{k\in\zz}b^k\chi_{B_{k+1}\setminus B_k}(x)\ {\rm
if} \ x\ne \mathbf{0},\hs {\mathrm {or\ else}\hs } \rho(\mathbf{0}):=0.
\end{equation*}
By \eqref{e2.3}, we know that, for all $x,\,y\in\rn$,
$$\rho(x+y)\le
b^\sz\lf(\max\lf\{\rho(x),\,\rho(y)\r\}\r)\le b^\sz[\rho(x)+\rho(y)].$$

Now we recall the definition of mixed-norm Lebesgue space. Let $\vec{p}:=(p_1,\ldots,\,p_n)\in(0,\,\infty]^n$.
The {\it mixed-norm Lebesgue space}
$L^{\vec{p}}$ is defined to be the set of all measurable functions $f$ such that
$$\|f\|_{L^{\vec{p}}}:= \lf\{\int_{\mathbb{R}}\cdots\lf[\int_{\mathbb{R}}|f(x_1,\,\ldots,x_n)|
^{p_1}\,dx_1\r]^{{p_2}/{p_1}}\cdots dx_n\r\}^{{1}/{p_n}}<\infty$$
with the usual modifications made when $p_i=\infty$ for some $i \in \{1, \ldots, n\}$.

For any $\vec{p}:=(p_1,\ldots,\,p_n)\in(0,\,\infty]^n$, let
\begin{eqnarray}\label{e2.5}
&&p_- :=\min \{p_1,\,\ldots,p_n\}\quad \mathrm{and} \ \ p_+ :=\max \{p_1,\,\ldots,p_n\}.
\end{eqnarray}

\begin{lemma}\label{r2.1}{\rm\cite[Lemma 3.4]{hlyy20}}
Let $\vec{p}\in(0,\,\infty]^n$.
 Then, for any
$r\in(0,\,\infty)$ and $f\in L^{\vec{p}}$,
$$\lf\|{|f|^r}\r\|_{L^{\vec{p}}}=\lf\|f\r\|^r_{L^{r\vec{p}}}.$$
In addition, for any $\mu\in\ccc$, $\gamma\in [0,\,\min\{1,\,p_-\}]$ and $f,g\in L^{\vec{p}}$,
$\lf\|\mu f\r\|_{L^{\vec{p}}}=|\mu|\lf\|f\r\|_{L^{\vec{p}}}$ and
$$\lf\|f+g\r\|^{\gamma}_{L^{\vec{p}}}\leq\lf\|f\r\|^{\gamma}_{L^{\vec{p}}}+
\lf\|g\r\|^{\gamma}_{L^{\vec{p}}},$$
here and hereafter, for any $\alpha\in\mathbb{R}$, $\alpha\vec{p}:=(\alpha p_1,\ldots,\alpha p_n)$ and
\begin{align}\label{e2.5.1}
\underline{p}:=\min\{p_-,\,1\}
\end{align}
with $p_-$ as in \eqref{e2.5}.
\end{lemma}

A $C^\infty$ function $\varphi$ is said to belong to the Schwartz class $\cs$ if,
for every integer $\ell\in\zz_+$ and multi-index $\alpha$,
$\|\varphi\|_{\alpha,\ell}:=\dsup_{x\in\rn}[\rho(x)]^\ell|\partial^\az\varphi(x)|<\infty$.
The dual space of $\cs$, namely, the space of all tempered distributions on $\rn$ equipped with the weak-$\ast$
topology, is denoted by $\cs'$. For any $N\in\zz_+$, let
\begin{eqnarray*}
\cs_N:=\lf\{\varphi\in\cs:\ \|\varphi\|_{\alpha,\ell}\leq1,\ |\alpha|\leq N,\ \ \ell\leq N\r\}.
\end{eqnarray*}
In what follows, for $\varphi\in \cs$, $k\in\zz$ and $x\in\rn$, let $\varphi_k(x):= b^{-k}\varphi\lf(A^{-k}x\r)$.

\begin{definition}
Let $\varphi\in \cs$ and $f\in \cs'$. The{\it{ non-tangential maximal function}} $M_{\varphi}(f)$ with respect to $\varphi$ is defined by setting,
for any $x\in\rn$,
\begin{eqnarray*}
M_{\varphi}(f)(x):=\sup_{y\in x+B_k, k\in\zz}|f*\varphi_{k}(y)|.
\end{eqnarray*}
The $\it{radial\ maximal\ function}$ $M^0_\varphi(f)$ with respect to $\varphi$ is defined by
setting, for any $x\in\rn$,
$$M^0_\varphi(f)(x):=\sup_{k\in\zz}|f\ast\varphi_k(x)|.$$
Moreover, for any given $N\in \nn$, the {\it non-tangential grand maximal function} $M_{N}(f)$ of $f\in \cs'$ is defined by setting,
for any $x\in\rn$,
\begin{eqnarray*}
M_{N}(f)(x):=\sup_{\varphi\in \cs_N}M_{\varphi}(f)(x).
\end{eqnarray*}
The {\it radial grand maximal function} $M^0_{N}(f)(x)$ of $f\in\cs'$ is defined by setting, for any $x\in\rn$,
$$M^0_{N}(f)(x):=\sup_{\varphi\in\mathcal{S}_{N}}M^0_\varphi (f)(x).$$
\end{definition}

The following anisotropic mixed-norm Hardy space $H^{\vec{p}}_{A}$ was introduced in \cite[Definition 2.5]{hlyy20}.
\begin{definition}\label{d2.4}
Let $\vec{p}\in (0,\,\infty)^n$, $A$ be a dilation and $N\in\nn\cap[\lfloor({1/\underline{p}}-1)\ln b/\ln\lambda_{-}\rfloor+2,\,\infty)$, where $\underline{p}$
is as in \eqref{e2.5.1}. The{\it{ anisotropic mixed-norm Hardy space}} $H_{A}^{\vec{p}}$ is defined as
\begin{eqnarray*}
H_{A}^{\vec{p}}:=\lf\{f\in \cs':M_{N}(f)\in L^{\vec{p}}\r\}
\end{eqnarray*}
and, for any $f\in H_{A}^{\vec{p}}$, let $\|f\|_{H_{A}^{\vec{p}}}:=\|M_{N}(f)\|_{L^{\vec{p}}}$.
\end{definition}
\begin{remark} Let $\vec{p}\in (0,\,\infty)^n$.
\begin{enumerate}
\item[\rm{(i)}]
From \cite[Theorem 4.7]{hlyy20}, we know that the $H_{A}^{\vec{p}}$ is independent of the choice of $N$, as long as $N\in\nn\cap[\lfloor({1/\underline{p}}-1)\ln b/\ln\lambda_{-}\rfloor+2,\,\infty)$.
\item[\rm{(ii)}]  When
 $\vec{p}:=\overbrace{\{p,\,\ldots,p\}}^{n\,\mathrm{times}}$, where $p\in(0,\,\infty)$, the space $H^{\vec{p}}_{A}$ is reduced to the anisotropic Hardy $H^{p}_{A}$ studied in \cite[Definition 3.11]{b03}.
\item[\rm{(iii)}] From \cite[Proposition 4]{hlyy20}, we know that,
 when $$A:=\left(
                    \begin{array}{cccc}
                      2^{a_1} & 0 & \cdots & 0 \\
                      0 & 2^{a_2} & \cdots & 0 \\
                      \vdots & \vdots &  & \vdots \\
                      0 & 0 & \cdots & 2^{a_n} \\
                    \end{array}
                  \right)
$$
with $\vec{a}:=(a_1,\ldots,a_n)\in[1,\,\infty)^n$,
the space $H^{\vec{p}}_{A}$ is reduced to the anisotropic Hardy $H^{\vec{p}}_{\vec{a}}$ studied in \cite{cgn17,hlyy19}.
\end{enumerate}
\end{remark}
We recall
the following notion of anisotropic mixed-norm $(\vec{p},\,q,\,s)$-atoms introduced in \cite[Definition 4.1]{hlyy20}.

\begin{definition}\label{d3.1}
Let $\vec{p}\in(0,\,\infty)^n$, $q\in(1,\,\infty]$ and
$s\in[\lfloor(1/{p_-}-1) {\ln b/\ln \lambda_-}\rfloor,\,\infty)\cap\zz_+$ with $p_-$ as in \eqref{e2.5}. An {\it anisotropic mixed-norm $(\vec{p},\,q,\,s)$-atom} is a measurable function $a$ on $\rn$ satisfying
\begin{enumerate}
\item[\rm{(i)}] (support) $\supp a\subset B$, where $B\in\mathfrak{B}$ and $\mathfrak{B}$ is as in \eqref{e2.1};
\item[\rm{(ii)}] (size) $\|a\|_{L^q}\le \frac{|B|^{1/q}}{\|\chi_B\|_{L^{\vec{p}}}}$;
\item[\rm{(iii)}] (vanishing moment) $\int_\rn a(x)x^\alpha dx=0$ for any $\alpha\in \mathbb{Z}^n_+$ with $|\alpha|\leq s$.
\end{enumerate}
\end{definition}
In this paper, we call an anisotropic mixed-norm $(\vec{p},\,q,\,s)$-atom simply by
a $(\vec{p},\,q,\,s)$-atom. The following anisotropic mixed-norm atomic Hardy space
was introduced in \cite{hlyy20}.
\begin{definition}
Let $\vec{p}\in (0,\,\infty)^n$, $q\in(1,\,\infty]$, $A$ be a dilation and
$s\in[\lfloor(1/{p_-}-1) {\ln b/\ln \lambda_-}\rfloor,\,\infty)\cap\zz_+$ with $p_-$ as in \eqref{e2.5}.
 The {\it anisotropic mixed-norm atomic Hardy space}
$H^{\vec{p},\,q,\,s}_{A}$
is defined to be the set of all distributions $f\in \cs'$ satisfying that there
exist $\{\lambda_i\}_{i\in\nn}\subset\ccc$ and a sequence of
$(\vec{p},\,q,\,s)$-atoms, $\{a_i\}_{i\in\nn}$, supported, respectively,
on $\{{B^{(i)}}\}_{i\in\nn}\subset\mathfrak{B}$ such that
\begin{align*}
f=\sum_{i\in\nn} \lz_{i}a_i \ \ \mathrm{in\ } \ \cs'.
\end{align*}
Moreover, for any $f\in H^{\vec{p},\,q,\,s}_{A}$, let
$$\|f\|_{H^{\vec{p},\,q,\,s}_{A}}
:=\inf \lf\|\lf\{\sum_{i\in\nn} \lf[\frac{|\lambda_i|\chi_{{B^{(i)}}}}{\|\chi_{{B^{(i)}}}\|_{L^{\vec{p}}}}\r]
^{\eta}\r\}^{1/\eta}\r\|
_{L^{\vec{p}}},$$
where $\eta\in(0,\,\min\{1,\,p_{-}\})$ and the infimum is taken over all the decompositions of $f$ as above.
\end{definition}

\begin{lemma}\label{l3.1}{\rm\cite[Theorem 4.7]{hlyy20}}
Let $\vec{p}\in (0,\,\infty)^n$, $q\in(\max\{p_+,\,1\},\,\infty]$
with $p_+$ as in \eqref{e2.5},
$s\in[\lfloor(1/{p_-}-1) {\ln b/\ln \lambda_-}\rfloor,\,\infty)\cap\zz_+$ with $p_-$ as in \eqref{e2.5} and
$N\in\nn\cap[\lfloor(1/{\underline{p}}-1) {\ln b/\ln \lambda_-}\rfloor+2,\,\infty)$.
Then $$H^{\vec{p}}_{A}=H^{\vec{p},\,q,\,s}_{A}$$ with equivalent quasi-norms.
\end{lemma}


We recall the definition of anisotropic Hardy-Littlewood maximal function $M_{\mathrm{HL}}(f)$. For any $f\in L_{\mathrm{loc}}^1$ and $x\in \rn$,
\begin{align}\label{e2.9}
M_{\mathrm{HL}}(f)(x):=\sup_{x\in B\in\mathfrak{B}}\frac{1}{|B|}\int_{B}|f(z)|\,dz,
\end{align}
where $\mathfrak{B}$ is as in \eqref{e2.1}.

\begin{lemma}\label{l3.5}{\rm\cite[Lemma 4.4]{hlyy20}}
Let $\vec{p}\in (1,\,\infty)^n$ and  $u\in(1,\,\infty]$.
Then there exists a positive constant $C$ such that, for any sequence $\{f_k\}_{k\in\nn}$ of measurable functions,
$$\lf\|\lf\{\sum_{k\in\nn}\lf[M_{\mathrm{HL}}(f_k)\r]^u\r\}^{1/u}\r\|_{L^{\vec{p}}}\leq C
\lf\|\lf(\sum_{k\in\nn}|f_k|^u\r)^{1/u}\r\|_{L^{\vec{p}}}$$
with the usual modification made when $u=\infty$, where $M_{\mathrm{HL}}$ denotes the Hardy-Littlewood maximal operator as in \eqref{e2.9}.
\end{lemma}

\section{Atomic Decomposition of $H^{\vec{p}}_A$\label{s3}}
In this section, we obtain the new atomic decomposition of $H^{\vec{p}}_A$. Let
$$\mathcal{S}_{\infty}:= \lf\{\varphi \in\mathcal{S} : \int_{\rn}x^{\alpha}\varphi(x)dx = 0, \forall \alpha\in \mathbb{Z}^n_+\r\}$$
equipped with the same topology as $\mathcal{S}$, and $\mathcal{S}_{\infty}'$ be its dual space equipped with the
weak-$\ast$ topology. Now we introduce the new anisotropic mixed-norm atomic Hardy space
$\mathbb{H}^{\vec{p},\,q,\,s}_{A}$ in term of $\mathcal{S}_{\infty}'$.
\begin{definition}\label{d3.x}
Let $\vec{p}\in (0,\,\infty)^n$, $q\in(1,\,\infty]$, $A$ be a dilation and
$s\in[\lfloor(1/{p_-}-1) {\ln b/\ln \lambda_-}\rfloor,\,\infty)\cap\zz_+$ with $p_-$ as in \eqref{e2.5}.
 The {\it anisotropic mixed-norm atomic Hardy space}
$\mathbb{H}^{\vec{p},\,q,\,s}_{A}$
is defined to be the set of all distributions $f\in \cs'$ satisfying that there
exist $\{\lambda_i\}_{i\in\nn}\subset\ccc$ and a sequence of
$(\vec{p},\,q,\,s)$-atoms, $\{a_i\}_{i\in\nn}$, supported, respectively,
on $\{{B^{(i)}}\}_{i\in\nn}\subset\mathfrak{B}$ such that
\begin{align*}
f=\sum_{i\in\nn} \lz_{i}a_i \ \ \mathrm{in\ } \ \mathcal{S}_{\infty}'.
\end{align*}
Moreover, for any $f\in \mathbb{H}^{\vec{p},\,q,\,s}_{A}$, let
$$\|f\|_{\mathbb{H}^{\vec{p},\,q,\,s}_{A}}
:=\inf \lf\|\lf\{\sum_{i\in\nn} \lf[\frac{|\lambda_i|\chi_{{B^{(i)}}}}{\|\chi_{{B^{(i)}}}\|
_{L^{\vec{p}}}}\r]^{\eta}\r\}^{1/\eta}\r\|
_{L^{\vec{p}}},$$
where $\eta\in(0,\,\min\{1,\,p_{-}\})$ and the infimum is taken over all the decompositions of $f$ as above.
\end{definition}

\begin{theorem}\label{t3.10}
Let $\vec{p}\in (0,\,1]^n$, $q\in(1,\,\infty]$,
$s\in[\lfloor(1/{p_-}-1) {\ln b/\ln \lambda_-}\rfloor,\,\infty)\cap\zz_+$ with $p_-$ as in \eqref{e2.5} and
$N\in\nn\cap[\lfloor(1/{\underline{p}}-1) {\ln b/\ln \lambda_-}\rfloor+2,\,\infty)$.
Then $$H^{\vec{p}}_{A}=\mathbb{H}^{\vec{p},\,q,\,s}_{A}$$
in the following sense: if $f\in H^{\vec{p}}_{A}$, then $f \in\mathbb{H}^{\vec{p},\,q,\,s}_{A}$ and
$$\|f\|_{\mathbb{H}^{\vec{p},\,q,\,s}_{A}}\leq\|f\|_{H^{\vec{p}}_{A}}.$$
Conversely, if $f \in\mathbb{H}^{\vec{p},\,q,\,s}_{A}$, then there exists a unique extension $F\in H^{\vec{p}}_{A}$ such that,
for any $\phi\in\mathcal{S}_{\infty}$, $\langle F,\phi\rangle =\langle f,\phi\rangle$ and
$$\|F\|_{H^{\vec{p}}_{A}}\leq\|f\|_{\mathbb{H}^{\vec{p},\,q,\,s}_{A}}.$$
\end{theorem}
\begin{remark}
Let $\vec{p}\in (0,\,\infty)^n$.
\begin{enumerate}
\item[(i)]
  When
 $\vec{p}:=\overbrace{\{p,\,\ldots,p\}}^{n\,\mathrm{times}}$, where $p\in(0,\,1]$, this result is also new.
\item[(ii)] When $$A:=\left(
                    \begin{array}{cccc}
                      2^{a_1} & 0 & \cdots & 0 \\
                      0 & 2^{a_2} & \cdots & 0 \\
                      \vdots & \vdots &  & \vdots \\
                      0 & 0 & \cdots & 2^{a_n} \\
                    \end{array}
                  \right)
$$
with $\vec{a}:=(a_1,\ldots,a_n)\in[1,\,\infty)^n$, this result is reduced to the \cite[Theorem 3.3]{hlyy20x}.
\end{enumerate}
\end{remark}

To prove Theorem \ref{t3.10}, we need some technical lemmas as following.

\begin{lemma}\label{l2.61}
Let $\vec{p}\in (0,\,\infty)^n$, $q\in(1,\infty)$ and $s\in[\lfloor(1/{p_-}-1){\ln b/\ln \lambda_-}\rfloor,\,\infty)\cap\zz_+$ with $p_-$ as in \eqref{e2.5}. Then there exists a positive constant $C$ such that, for any $(\vec{p},q,s)$-atom b and $\varphi \in \mathcal{S}$,
$$\lf|\int_{\rn}b(y)\varphi(y)dy\r|\leq C M_{s}(\varphi),$$
where $M_{s}(\varphi):=\sup\{\|\varphi\|_{\alpha,\,0}:\alpha \in \mathbb{Z}^{n}_{+},  |\alpha|\leq s+1\}.$
\end{lemma}
\begin{proof} For any $(\vec{p},q,s)$-atom $b$, we assume supp $b\subset B_{k}.$
Now we claim that, for any $k\in\zz$,
$$\|\chi_{B_k}\|_{L^{\vec{p}}}^{-1}\lesssim\max\{b^{-k/p_-},\,b^{-k/p_+}\}.$$
In fact, there exists a $K\in\zz$ large enough such that, when $k\in(K,\,\infty)\cap\zz$, the
\begin{align*}
\|\chi_{B_k}\|_{L^{\vec{p}}}&= \lf\{\int_{\mathbb{R}}\cdots\lf[\int_{\mathbb{R}}|\chi_{B_k}(x_1,\,\ldots x_n)|^{p_1}\,dx_1\r]^{{p_2}/{p_1}}\cdots dx_n\r\}^{{1}/{p_n}}\\
&\geq\lf\{\int_{\mathbb{R}}\cdots\lf[\int_{\mathbb{R}}|\chi_{B_k}(x_1,\,\ldots x_n)|^{p_+}\,dx_1\r]^{{p_+}/{p_+}}\cdots dx_n\r\}^{{1}/{p_+}}\\
&=|B_k|^{1/p_+}.
\end{align*}
When $k\in(-\infty,\,K]$, by \cite[Lemma 6.8]{hlyy20}, we conclude that, for any $\nu\in(0, 1)$,
$$\frac{\|\chi_{B_K}\|_{L^{\vec{p}}}}{\|\chi_{B_k}\|_{L^{\vec{p}}}}
\lesssim b^{(1+\nu)(K-k)/p_-}.$$
Let $\nu\rightarrow 0$. Then
$$\frac{1}{\|\chi_{B_k}\|_{L^{\vec{p}}}}
\lesssim \frac{b^{K/p_-}}{\|\chi_{B_K}\|_{L^{\vec{p}}}}b^{-k/p_-}
\thicksim|B_k|^{-1/p_-}.$$
If $k>0$, we have
$$\|b\|_{L^{q}}\leq \frac{|B_{k}|^{{1}/{q}}}{\|\chi_{B_k}\|_{L^{\vec{p}}}}\lesssim|B_k|^{{1}/{q}-{1}/{p_{+}}}.$$
 If $k\leq 0$, we have
$$ \|b\|_{L^{q}}\leq \frac{|B_{k}|^{{1}/{q}}}{\|\chi_{B_k}\|_{L^{\vec{p}}}}\lesssim|B_k|^{{1}/{q}-{1}/{p_{-}}}.$$
Therefore, when $k>0$, we have
\begin{align*}
\lf|\int_{\mathbb{R}^{n}}b(y)\varphi(y)dy\r| &\leq \|\varphi\|_{L^{\infty}} \|b\|_{{L^{q}}}|B_k|^{{1}/{q'}}\\
&\leq \|\varphi\|_{\mathbf{0},\,0}\|b\|_{{L^{q}}}|B_k|^{{1}/{p_{+}}-{1}/{q}}  \\
&\leq M_{s}(\varphi).
\end{align*}
When $k\leq0$, by the vanishing moment of $b$ and the Taylor reminder theorem, we obtain
\begin{align*}
\lf|\int_{\mathbb{R}^{n}}b(y)\varphi(y)dy\r| &\leq \lf|\int_{\rn}b(y)\lf[\varphi(y)-\mathop {\sum }\limits_{|\alpha|\leq s}\frac{\partial^{\alpha}\varphi(\mathbf{0})}{\alpha!}y^{\alpha}\r]dy\r|\\
&\lesssim M_{s}(\varphi)\lf|\int_{B_{k}}b(y)|y|^{s+1}dy\r|\\
&\lesssim M_{s}(\varphi)\|b\|_{{L^{q}}}b^{k(s+1)}|B_{k}|^{{1}/{q'}}  \\
&\lesssim M_{s}(\varphi)|B_{k}|^{{1}/{q}-{1}/{p_{-}}}|B_{k}|^{s+1-{1}/{q}}\\
&\lesssim M_{s}(\varphi).
\end{align*}
This finishes the proof of Lemma \ref{l2.61}.
\end{proof}
\begin{lemma}\label{l3.1x}
Let $\vec{p}\in(0,\,1]^n$. Then, for any $\{\lambda_i\}_{i\in \nn}\subset\mathbb{C}$ and $\{B^{(i)}\}_{i\in\nn} \subset\mathfrak{B}$,
$$\sum_{i\in\nn}|\lambda_i|\leq\lf\|\lf\{\sum_{i\in\nn}
\lf[\frac{|\lambda_i|\chi_{B^{(i)}}}{\|\chi_{B^{(i)}}\|_{L^{\vec{p}}}}\r]
^{\eta}\r\}^{1/\eta}\r\|_{L^{\vec{p}}}$$
where $\eta\in(0,\,\underline{p})$ with $\underline{p}$ as in \eqref{e2.5.1}.
\end{lemma}
\begin{proof}
Since $\vec{p}\in(0,\,1]^n$ and the well-known inequality that, for all $\{\lambda_i\}_{i\in\nn}\subset\mathbb{C}$ and $\vartheta\in(0,\,1]$,
\begin{align}\label{e3.x}
\lf(\sum_{i\in\nn}|\lambda_i|\r)^{\vartheta}\leq
\sum_{i\in\nn}|\lambda_i|^{\vartheta}.
\end{align}
Then we have
 \begin{align*}
\sum_{i\in\nn}|\lambda_i|&=\sum_{i\in\nn}|\lambda_i| \nonumber
\lf\|\frac{\chi_{B^{(i)}}}{\|\chi_{B^{(i)}}\|_{L^{\vec{p}}}}\r\|_{{L^{\vec{p}}}}\\ \nonumber
&\leq\lf\|\sum_{i\in\nn}\frac{|\lambda_i|\chi_{B^{(i)}}}{\|\chi_{B^{(i)}}\|_{L^{\vec{p}}}}
\r\|_{L^{\vec{p}}}\\ \nonumber
&\leq\lf\|\lf\{\sum_{i\in\nn}\lf[\frac{|\lambda_i|\chi_{B^{(i)}}}
{\|\chi_{B^{(i)}}\|_{L^{\vec{p}}}}\r]
^{\eta}\r\}^{1/\eta}\r\|_{L^{\vec{p}}}.\nonumber
\end{align*}
This finishes the proof of Lemma \ref{l3.1x}.
\end{proof}

\begin{lemma}\label{l2.60}\rm{\cite{l07}}
Let $\pi:\,\mathcal{S}^{' }\rightarrow \mathcal{S}^{'}_{\infty}$ satisfy that, for any
$f\in \mathcal{S}^{' }$ and $\phi\in\mathcal{S}^{'}_{\infty}$,
$$\langle\pi(f),\,\phi\rangle=\langle f,\,\phi \rangle.$$ Then
$$\mathcal{P}=\{f\in\mathcal{S}^{' }:\pi(f)=0\}.$$
where $\mathcal{P}$ denote all polynomials on $\rn$.
Moreover, $\mathcal{P}$ is closed in $\mathcal{S}^{' }$.
\end{lemma}

\begin{proof}[Proof of Theorem \ref{t3.10}]
By the definitions of $\mathcal{S}$ and $\mathcal{S}_{\infty}$, we find that $$\mathcal{S}_{\infty}\subset\mathcal{S},$$ which implies that
$$\mathcal{S}^{'}\subset\mathcal{S}^{'}_{\infty}.$$
From this, Lemma \ref{l3.1}, we conclude that
$${H^{\vec{p}}_{A}}={H^{\vec{p},\,q,\,s}_{A}}\subset{\mathbb{H}^{\vec{p},\,q,\,s}_{A}}$$
and for any $f\in {H^{\vec{p}}_{A}}$,
$$\|f\|_{\mathbb{H}^{\vec{p},\,q,\,s}_{A,}}\leq\|f\|_{H^{\vec{p}}_{A}}$$
Therefore, to prove Theorem \ref{t3.10}, it suffices to show that, for any $f\in \mathbb{H}^{\vec{p},\,q,\,s}_{A}$,
there exists a unique extension $F\in H^{\vec{p},\,q,\,s }_{A}$ such that, for any $\phi \in \mathcal{S}_{\infty}$.
$$\langle F,\,\phi\rangle=\langle f,\,\phi\rangle$$
and
$$\|F\|_{H^{\vec{p},\,q,\,s}_{A}}\leq\|f\|_{\mathbb{H}^{\vec{p},\,q,\,s}_{A}}.$$
To show this, for any $f\in \mathbb{H}^{\vec{p},\,q,\,s }_{A}$, by Definition \ref{d3.x},
we deduce that there exists a sequence of $(\vec{p},\,q,\,s)$-atom $\{a_{i}\}_{i\in\nn}$,
supported, respectively, on  $\{B^{(i)}\}_{i\in\nn}\subset\mathfrak{B}$, such that
$$f=\sum_{i\in\nn} \lz_ia_i\,\,\mathrm{in} \,\,\mathcal{S}^{'}_{\infty}$$
and
\begin{align*}
\|f\|_{\mathbb{H}^{\vec{p},\,q,\,s}_{A}}
\thicksim \lf\|\lf\{\sum_{i\in\nn} \lf[\frac{|\lambda_i|\chi_{{B^{(i)}}}}{\|\chi_{{B^{(i)}}}
\|_{L^{\vec{p}}}}\r]^{\eta}\r\}^{1/\eta}\r\|
_{L^{\vec{p}}}<\infty
\end{align*}
with $\eta\in(0,\,\min\{1,\,p_-\})$.

For any $\varphi\in\mathcal{S}$, set
$$\langle F,\,\varphi\rangle:=\displaystyle\sum_{i\in\mathbb{N}}\lambda_i\int_{\rn}a_i(x)\varphi(x)\,dx.$$
Then, using Lemmas \ref{l2.61} and \ref{l3.1x}, we have
$$|\langle F,\,\varphi\rangle|\lesssim M_{s}(\varphi)\sum_{i\in\mathbb{N}}|\lambda_i|
\lesssim M_{s}(\varphi)\lf\|\lf\{\sum_{i\in\nn} \lf[\frac{|\lambda_i|\chi_{{B^{(i)}}}}{\|\chi_{{B^{(i)}}}
\|_{L^{\vec{p}}}}\r]^{\eta}\r\}^{1/\eta}\r\|
_{L^{\vec{p}}}, \,\, \eta\in(0,\,\min\{1,\,p_-\}),$$
which implies  that $F\in\mathcal{S'}$.
From this, we see that, $F\in H^{\vec{p},\,q,\,s }_{A}$ and for all $\phi\in\mathcal{S}_{\infty}$,
$$\langle F,\,\phi\rangle=\langle f,\,\phi\rangle.$$
By the Definition \ref{d3.x}, we know that
$$\|F\|_{H^{\vec{p},\,q,\,s}_{A}}\leq\|f\|_{\mathbb{H}^{\vec{p},\,q,\,s}_{A}}.$$
Finally, we only need to show that the distribution $F$ is unique.
In fact, if there exists distribution $\tilde{F}\in H^{\vec{p},\,q,\,s }_{A}$ such that,
for any $\psi\in\mathcal{S}_{\infty}$, $\langle\tilde{F},\,\psi\rangle=\langle f,\,\psi\rangle$.
Then, by Lemma \ref{l2.60}, we obtain
$$F-\tilde{F}\in \mathcal{P}.$$
By the fact that any element of $H^{\vec{p},\,q,\,s}_{A}$ vanishes weakly in infinity
 and $F-\tilde{F}\in H^{\vec{p},\,q,\,s}_{A}$, we deduce that
$F=\tilde{F}$ in $\mathcal{S}_{\infty}'$. This finishes the proof of Theorem \ref{t3.10}.

\end{proof}

\section{Molecular Decomposition of $H^{\vec{p}}_A$\label{s4}}
\hskip\parindent
In this section, we introduce the definition of anisotropic  mixed-norm molecules as follows.
\begin{definition}\label{d2.4x}
Let $\vec{p}\in (0,\,\infty)^n$, $q\in(1,\,\infty]$,
\begin{eqnarray}\label{e2.10}
s\in\lf[\lf\lfloor\lf(\frac{1}{p_-}-1\r)\frac{\ln b}{\ln\lambda_-}\r\rfloor,\,\infty\r)\cap\zz_+.
\end{eqnarray}
and $\varepsilon\in(0,\,\infty)$. A measurable function $\mathfrak{M}$ is called an
{\it{anisotropic mixed-norm $(\vec{p},\,q,\,s,\,\ez)$-molecule}} associated with a dilated ball
$x_0+B_i\in\mathfrak{B}$ if
\begin{enumerate}
\item[\rm{(i)}]  for each
$j\in\zz_+$, $\|\mathfrak{M}\|_{L^q(U_j(x_0+B_i))}\le b^{-j\varepsilon}\frac{|B_{i}|^{1/q}}{\|\chi_{x_0+B_i}\|_{L^{\vec{p}}}}$, where $U_0(x_0+B_i):=x_0+B_i$ and, for each $j\in\nn$, $U_j(x_0+B_{i}):=x_0+(A^jB_{i})\setminus(A^{j-1}B_{i})$;
\item[\rm{(ii)}] for all $\alpha\in \mathbb{Z}^n_+$ with $|\alpha|\leq s$, $\int_\rn \mathfrak{M}(x)x^\alpha dx=0$.
\end{enumerate}
\end{definition}
\begin{remark}\label{r2.1x}
Let $\vec{p}\in (0,\,\infty)^n$.
\begin{enumerate}
\item[(i)]
When
 $\vec{p}:=\overbrace{\{p,\,\ldots,p\}}^{n\,\mathrm{times}}$, where $p\in(0,\,1]$,
the definition of
 the molecule in Definition \ref{d2.4x} is reduced to the molecule in \cite[Definition 3.7]{lyy16}.
\item[(ii)] When it comes back to the isotropic setting, i.e., $A:=2{\rm I}_{n\times n}$, and $\rho(x):= |x|^n$ for all $x\in\rn$, the definition of the molecule in Definition \ref{d2.4} is also new.

\end{enumerate}
\end{remark}

In what follows, we call an anisotropic mixed-norm $(\vec{p},\,q,\,s,\,\varepsilon)$-molecule simply by
$(\vec{p},\,q,\,s,\,\varepsilon)$-molecule. Via $(\vec{p},\,q,\,s,\,\varepsilon)$-molecules,
we introduce the following anisotropic mixed-norm molecular Hardy space
$H^{\vec{p},\,q,\,s,\,\varepsilon}_{A,\,\mathrm{mol}}$.


\begin{definition}\label{d2.5}
Let $\vec{p}\in (0,\,\infty)^n$, $q\in(1,\,\infty]$, $A$ be a dilation,
 $s$ be as in \eqref{e2.10} and $\varepsilon\in(0,\,\infty)$.
 The {\it anisotropic mixed-norm molecular Hardy space}
$H^{\vec{p},\,q,\,s,\,\varepsilon}_{A,\,\mathrm{mol}}$
is defined to be the set of all distributions $f\in \cs'$ satisfying that there
exist $\{\lambda_i\}_{i\in\nn}\subset\ccc$ and a sequence of
$(\vec{p},\,q,\,s,\,\varepsilon)$-molecules, $\{\mathfrak{M}_i\}_{i\in\nn}$, associated, respectively,
with $\{{B^{(i)}}\}_{i\in\nn}\subset\mathfrak{B}$ such that
\begin{align*}
f=\sum_{i\in\nn} \lz_{i}\mathfrak{M}_i \ \ \mathrm{in\ } \ \cs'.
\end{align*}
Moreover, for any $f\in H^{\vec{p},\,q,\,s,\,\varepsilon}_{A,\,\mathrm{mol}}$, let
$$\|f\|_{H^{\vec{p},\,q,\,s,\,\varepsilon}_{A,\,\mathrm{mol}}}
:=\inf \lf\|\lf\{\sum_{i\in\nn} \lf[\frac{|\lambda_i|\chi_{{B^{(i)}}}}{\|\chi_{{B^{(i)}}}\|_{L^{\vec{p}}}}\r]
^{\eta}\r\}^{1/\eta}\r\|
_{L^{\vec{p}}},$$
where $\eta\in(0,\,\min\{1,\,p_{-}\})$ and the infimum is taken over all the decompositions of $f$ as above.
\end{definition}

The following Theorem \ref{t2.6} shows the molecular characterization of
$H^{\vec{p}}_{A}$, whose proof is given in the next section.
\begin{theorem}\label{t2.6}
Let $\vec{p}\in (0,\,\infty)^n$ and $q\in(1,\,\infty]\cap(p_+,\,\infty]$ with $p_+$ as in \eqref{e2.5},
$s$ be as in \eqref{e2.10}, $\varepsilon\in (\max\{1,\,(s+1)\log_b(\lambda_+)\},\,\infty)$ and $N\in\nn\cap[\lfloor(1/{\underline{p}}-1) {\ln b/\ln \lambda_-}\rfloor+2,\,\infty)$
with $\underline{p}$ as in \eqref{e2.5.1}.
Then $$H^{\vec{p}}_{A}=H^{\vec{p},\,q,\,s,\,\varepsilon}_{A,\,\mathrm{mol}}$$ with equivalent quasi-norms.
\end{theorem}
\begin{remark}\label{r2.7}
Let $\vec{p}\in (0,\,1]^n$.
\begin{enumerate}
\item[(i)]Liu et al. \cite{lyy17} introduced the anisotropic Hardy-Lorentz space $H^{{p},\,q}
_A$, where $p\in(0,\,1]$ and $q\in(0,\infty]$.
When $\vec{p}:=\overbrace{\{p,\,\ldots,p\}}^{n\,\mathrm{times}}$ with $p\in(0,\,1]$,  the
molecular characterization of $H^{\vec{p}}_{A}$ in Theorem \ref{t2.6} is reduced to the molecular characterization of anisotropic Hardy spaces $H^{p}_{A}$=$H^{p,\,p}_{A}$ in \cite[Theorem 3.9]{lyy16}.
\item[(ii)] When it comes back to the isotropic setting, i.e., $A:=2{\rm I}_{n\times n}$, the
 molecular characterization of $H^{\vec{p}}_{A}$ in Theorem \ref{t2.6} is still new.
\end{enumerate}
\end{remark}

To show Theorem \ref{t2.6}, we need the following lemma.
By a similar proof of \cite[Lemma 4.5]{hlyy20}, we obtain the  following useful conclusion; the  details are omitted.
\begin{lemma}\label{l4.1x}
Let $\vec{p}\in (0,\,\infty)^n$ and $q\in(1,\,\infty]\cap(p_{+},\,\infty]$ with $p_+$ as in \eqref{e2.5}. Assume that
$\{\lambda_{i}\}_{i\in\nn}\subset\ccc$, $\{B^{(i)}\}_{i\in\nn}\subset\mathfrak{B}$ and $\{a_{i}\}_{i\in\nn}\in L^q$
satisfy, for any $i\in\mathbb{N}$, \,$\mathrm{supp}\, a_i\subset A^{j_0}B^{(i)}$ with some fixed $j_0\in\zz$,
$\|a\|_{L^q}\le \frac{|B^{(i)}|^{1/q}}{\|\chi_{B^{(i)}}\|_{L^{\vec{p}}}}$
and
$$\lf\|\lf\{\sum_{i\in\nn} \lf[\frac{|\lambda_i|\chi_{{B^{(i)}}}}{\|\chi_{{B^{(i)}}}\|_{L^{\vec{p}}}}\r]
^{\eta}\r\}^{1/\eta}\r\|
_{L^{\vec{p}}}<\infty.$$
Then
$$\lf\|\lf[\sum_{i\in\nn}|\lambda_{i}a_{i}|^{\eta}\r]^{1/\eta}\r\|
_{L^{\vec{p}}}\leq C\lf\|\lf\{\sum_{i\in\nn} \lf[\frac{|\lambda_i|\chi_{{B^{(i)}}}}{\|\chi_{{B^{(i)}}}\|_{L^{\vec{p}}}}\r]
^{\eta}\r\}^{1/\eta}\r\|
_{L^{\vec{p}}},$$
where $\eta\in(0,\,\min\{1,\,p_{-}\})$ and $C$ is a positive constant independent of $\{\lambda_{i}\}_{i\in\nn}$, $\{B^{(i)}\}_{i\in\nn}$ and $\{a_{i}\}_{i\in\nn}$.
\end{lemma}



\begin{proof}[Proof of Theorem \ref{t2.6}]
By the definitions of $(\vec{p},\,q,\,s)$-atom and $(\vec{p},\,q,\,s,\,\varepsilon)$-molecule,
we find that a $(\vec{p},\,\infty,\,s)$-atom is also a $(\vec{p},\,q,\,s,\,\varepsilon)$-molecule, which implies that
$$H^{\vec{p},\,\infty, \,s}_{A}\subset H^{\vec{p},\,q,\,s,\,\varepsilon}_{A,\,\mathrm{mol}}.$$
This, combined with Lemma \ref{l3.1}, further implies that, to prove Theorem \ref{t2.6}, it suffices to show
$H^{\vec{p},\,q,\,s,\,\varepsilon}_{A,\,\mathrm{mol}}\subset H^{\vec{p}}_{A}$.

For any $f\in H^{\vec{p},\,q,\,s,\,\varepsilon}_{A,\,\mathrm{mol}}$, by Definition \ref{d2.5},
we have that there exists a sequence of $(\vec{p},\,q,\,s,\,\varepsilon)$-molecules,
$\{\mathfrak{M}_{i}\}_{i\in\nn}$, associated with dilated balls $\{B^{(i)}\}_{i\in\nn}\subset\mathfrak{B}$,
where $B^{(i)}:=x_i+B_{\ell_i}$ with $x_i\in\rn$ and $\ell_i\in\zz$, such that
 $$f=\sum_{i\in\nn} \lz_i\mathfrak{M}_{i}\,\,\mathrm{in} \,\,\mathcal{S}^{'}$$
and
\begin{align}\label{e3.3}
\|f\|_{H^{\vec{p},\,q,\,s,\,\ez}_{A,\,\mathrm{mol}}}
\thicksim \lf\|\lf\{\sum_{i\in\nn} \lf[\frac{|\lambda_i|\chi_{{B^{(i)}}}}{\|\chi_{{B^{(i)}}}\|_{L^{\vec{p}}}}\r]
^{\eta}\r\}^{1/\eta}\r\|
_{L^{\vec{p}}},
\end{align}
where $\eta\in(0,\,\min\{1,\,p_{-}\})$.
To prove $f\in H^{\vec{p}}_{A}$,
it is easy to see that, for any $N\in\mathbb{N}\cap[\lfloor({1/\underline{p}}-1)\ln b/\ln\lambda_{-}\rfloor+2,\,\infty)$,
\begin{align*}
\|M_{N}(f)\|_{L^{\vec{p}}}^{\underline{p}}&=\lf\|M_{N}\lf(\sum_{i\in \nn}\lambda_{i}\mathfrak{M}_{i}\r)\r\|_{L^{\vec{p}}}^{\underline{p}}
\leq\lf\|\sum_{i\in \nn}|\lambda_{i}|M_{N}(\mathfrak{M}_{i})\r\|_{L^{\vec{p}}}^{\underline{p}}\\
&\leq\lf\|\sum_{i\in \nn}|\lambda_{i}|M_{N}(\mathfrak{M}_{i})\chi_{A^{2\sigma}B^{(i)}}\r\|_{L^{\vec{p}}}^{\underline{p}}+\lf\|\sum_{i\in \nn}|\lambda_{i}|M_{N}(\mathfrak{M}_{i})\chi_{({A^{2\sigma}B^{(i)}})^{\complement}}\r\|_{L^{\vec{p}}}^{\underline{p}}\\
&\leq\lf\|\lf\{\sum_{i\in \nn}\lf[|\lambda_{i}|M_{N}(\mathfrak{M}_{i})\chi_{A^{2\sigma}B^{(i)}}\r]
^{\eta}\r\}^{1/\eta}\r\|_{L^{\vec{p}}}^{\underline{p}}+\lf\|\sum_{i\in \nn}|\lambda_{i}|M_{N}(\mathfrak{M}_{i})\chi_{({A^{2\sigma}B^{(i)}})^{\complement}}\r\|_{L^{\vec{p}}}^{\underline{p}}\\
&=:\mathrm{I_{1}}+\mathrm{I_{2}},
\end{align*}
where $\eta\in(0,\,\min\{1,\,p_{-}\})$ and $A^{2\sigma}B^{(i)}$ is the $A^{2\sigma}$ concentric expanse on $B^{(i)}$, that is, $A^{2\sigma}B^{(i)}:=x_i+A^{2\sigma}B_{\ell_i}$.

To estimate $\mathrm{I_{1}}$, for any $q\in((\max\{p_+,\,1\},\,\infty)$, by
 the boundedness of $M_{N}$ on $L^{q}$ for all $q\in((\max\{p_+,\,1\},\,\infty)$
and H\"{o}lder's inequality,
we have
\begin{align*}
\|M_N(\mathfrak{M}_{i})\|_{L^q}\lesssim\|\mathfrak{M}_{i}\|_{L^q}
\lesssim&\sum_{j\in\nn}\frac{b^{-j\varepsilon}|B^{(i)}|^{1/q}}{\|\chi_{B^{(i)}}\|_{L^{\vec{p}}}}
\thicksim\frac{|B^{(i)}|^{1/q}}{\|\chi_{B^{(i)}}\|_{L^{\vec{p}}}}.
\end{align*}
By this, Lemma \ref{l4.1x}, $q\in((\max\{p_+,\,1\},\,\infty)$ and \eqref{e3.3}, we obtain
\begin{align*}
\mathrm{I_{1}}
&=\lf\|\lf\{\sum_{i\in\nn}\lf[\frac{|\lambda_i|}{\|\chi_{{B^{(i)}}}\|_
{L^{\vec{p}}}}\lf\|\chi_{B^{(i)}}\r\|_{L^{\vec{p}}}M_{N}(\mathfrak{M}_{i})
\chi_{A^{2\sigma}{B^{(i)}}}\r]^{\eta}\r\}^
{1/\eta}\r\|_{L^{\vec{p}}}^{\underline{p}}\\
&\lesssim\lf\|\lf\{\sum_{i\in\nn}\lf[\frac{|\lambda_i|}{\|\chi_{{B^{(i)}}}
\|_{L^{\vec{p}}}}\chi_{{B^{(i)}}}\r]^{\eta}\r\}^
{1/\eta}\r\|_{L^{\vec{p}}}^{\underline{p}}
\thicksim\|f\|_{H^{\vec{p},\,q,\,s,\,\ez}_{A,\,\mathrm{mol}}}^{\underline{p}}.
\end{align*}

To deal with ${\rm{I_{2}}}$, for any $i\in\nn$ and $x\in(x_i+A^{2\sigma} B_{\ell_i}))^\complement$, we need to estimate $M_{N}^0(\mathfrak{M}_i)(x)$, it's proof is similar to that of \cite[(3.48)]{lyy16}. Suppose that $P$ is a polynomial of degree not greater than $s$ which is determined later. By the H\"{o}lder inequality and Definitions \ref{d2.4x}, we know that
\begin{align}\label{e.t00}
&\lf|\mathfrak{M}_i*\varphi_k(x)\r|
=b^{-k}\lf|\int_{\rn}\mathfrak{M}_i(y)\varphi(A^{-k}(x-y))\,dy\r|\\\nonumber
\leq&b^{-k}\sum_{j=0}^{\infty}\lf|\int_{U_j(x_i+B_{\ell_i})}\mathfrak{M}_i(y)
\lf[\varphi(A^{-k}(x-y))-P(A^{-k}(x-y))\r]\,dy\r|\\\nonumber
\leq&b^{-k}\sum_{j=0}^{\infty}\lf\|\mathfrak{M}_i\r\|_{L^q(U_j(x_i+B_{\ell_i}))}
\lf|\int_{U_j(x_i+B_{\ell_i})}
\lf[\varphi(A^{-k}(x-y))-P(A^{-k}(x-y))\r]^{q'}\,dy\r|^{1/{q'}}\\\nonumber
\leq&b^{-k}\sum_{j=0}^{\infty}b^{-j\varepsilon}\frac{|x_i+B_{\ell_i}|^{1/q}}{\lf\|\chi_{x_i+B_{\ell_i}}
\r\|_{L^{\vec{p}}}}b^{(\ell_i+j)/{q'}}\sup_{z\in A^{-k}(x-x_i)+B_{\ell_i+j-k}}
\lf|\varphi(z)-P(z)\r|\\\nonumber
=&b^{\ell_i-k}\sum_{j=0}^{\infty}b^{j(1/{q'}-\varepsilon)}\frac{1}{\lf\|\chi_{x_i+B_{\ell_i}}
\r\|_{L^{\vec{p}}}}\sup_{z\in A^{-k}(x-x_i)+B_{\ell_i+j-k}}
\lf|\varphi(z)-P(z)\r|.\nonumber
\end{align}
Assume that $x\in x_i+B_{\ell_i+2\sigma+m+1}\backslash B_{\ell_i+2\sigma+m}$ for some $m\in\zz_+$. Then, by \eqref{e2.4}, we know that,
for any $k\in\zz$ and $j\in\zz_+$,
\begin{align}\label{e.t0x}
A^{-k}(x-x_i)+B_{\ell_i+j-k}
&\subset A^{-k}(B_{\ell_i+2\sigma+m+1}\backslash B_{\ell_i+2\sigma+m})+B_{\ell_i+j-k}\\\nonumber
&\subset A^{\ell_i+j-k}(B_{2\sigma+m+1}\backslash B_{2\sigma+m}+B_{0})\\\nonumber
&\subset A^{\ell_i+j-k}(B_m)^{\complement}.\nonumber
\end{align}
If $\ell_i\geq k$, let $P=0$. Then we have, for any $N\in\nn$,
\begin{align}\label{e.t0}
\sup_{z\in A^{-k}(x-x_i)+B_{\ell_i+j-k}}|\varphi(z)|
\lesssim&\sup_{z\in A^{\ell_i+j-k}(B_m)^{\complement}}\frac{1}{[1+\rho(z)]^N}\\\nonumber
\lesssim&b^{-N(\ell_i+j-k+m)}.
\end{align}
If $\ell_i<k$, let $P$ be the Taylor expansion of $\varphi$ at the point $A^{-k}(x-x_i)$ of order $s$. By Taylor's theorem, \eqref{e2.3v}, \eqref{e2.4v} and \eqref{e.t0x}, we obtain, for any $N\in\nn$,
\begin{align}\label{e.t1}
&\sup_{z\in A^{-k}(x-x_i)+A^{j-k}B_{\ell_i}}|\varphi(z)-P(z)|\\\nonumber
\lesssim&\sup_{y\in A^{j-k}B_{\ell_i}}\sup_{|\alpha|=s+1}
\lf|\partial^{\alpha}\varphi(A^{-k}(x-x_i)+y)\r||y|^{s+1}\\\nonumber
\lesssim&b^{j(s+1)\log_b^{\lambda_+}}\lambda_-^{(s+1)(\ell_i-k)}
\sup_{z\in A^{-k}(x-x_i)+A^{-k+j}B_{\ell_i}}\frac{1}{[1+\rho(z)]^{N}}\\\nonumber
\lesssim&b^{j(s+1)\log_b^{\lambda_+}}\lambda_-^{(s+1)(\ell_i-k)}
\sup_{z\in A^{\ell_i-k+j}(B_m)^\complement}\frac{1}{[1+\rho(z)]^{N}}\\\nonumber
\lesssim&b^{j(s+1)\log_b^{\lambda_+}}\lambda_-^{(s+1)(\ell_i-k)}
b^{-N(\ell_i-k+j+m)}.\nonumber
\end{align}
Notice that the supremum over $k\leq\ell_i$ has the largest value when $k=\ell_i$. Without loss of generality, we may
assume that $s:=\lfloor(\frac{1}{p_-}-1)\frac{\ln b}{\ln\lambda_-}\rfloor$ and $N := s + 2$. This implies $b\lambda_-^{s+1}
\leq b^N$ and the above supremum over $k>\ell_i$ is
attained when $\ell_i-k+j+m=0$.
By \eqref{e.t00}, \eqref{e.t0}, \eqref{e.t1} and the fact that $\varepsilon>(s+1)\log_b^{\lambda_+}$, we obtain
\begin{align*}
M_N^0(\mathfrak{M}_i)(x)=&\sup_{\varphi\in\mathcal{S}_N}\sup_{k\in\zz}|\mathfrak{M}_i*\varphi_k(x)|\\\nonumber
\lesssim&\frac{b^{\ell_i-k}}{\|\chi_{x_i+B_{\ell_i}}\|_{L^{\vec{p}}}}
\sum_{j=0}^{\infty}b^{(1/{q'}-\varepsilon)j}\max\lf\{b^{-N(\ell_i-k+j+m)},\r.\\ \nonumber
&\ \ \ \lf. b^{j(s+1)\log_b^{\lambda_+}}\lambda_-^{(s+1)(\ell_i-k)}
b^{-N(\ell_i-k+j+m)}\r\}\\\nonumber
\lesssim&\frac{b^{\ell_i-k}}{\|\chi_{x_i+B_{\ell_i}}\|_{L^{\vec{p}}}}
\sum_{j=0}^{\infty}b^{-j[\varepsilon-(s+1)\log_b^{\lambda_+}-1/{q'}+1]}
\max\lf\{b^{-mN},\,
(b\lambda_-^{s+1})^{-m}\r\}\\\nonumber
\lesssim&\frac{1}{\|\chi_{x_i+B_{\ell_i}}\|_{L^{\vec{p}}}}
b^{-m}b^{-m(s+1)(\ln\lambda_-/\ln b)}\\\nonumber
\lesssim&\frac{1}{\|\chi_{x_i+B_{\ell_i}}\|_{L^{\vec{p}}}}
b^{\ell_i[(s+1)(\ln\lambda_-/\ln b)+1]}b^{-(\ell_i+\sigma+m)[(s+1)(\ln\lambda_-/\ln b)+1]}\\\nonumber
\lesssim&\frac{1}{\|\chi_{x_i+B_{\ell_i}}\|_{L^{\vec{p}}}}
|x_i+B_{\ell_i}|^{[(s+1)(\ln\lambda_-/\ln b)+1]}[\rho(x-x_i)]^{-[(s+1)(\ln\lambda_-/\ln b)+1]}\\\nonumber
\thicksim&\lf\|\chi_{B^{(i)}}\r\|^{-1}_{L^{\vec{p}}}
\frac{\lf|{B^{(i)}}\r|^\theta}{[\rho(x-x_i)]^\theta}\\\nonumber
\lesssim&\lf\|\chi_{B^{(i)}}\r\|^{-1}_{L^{\vec{p}}}
\lf[M_{\mathrm{HL}}(\chi_{B^{(i)}})(x)\r]^\theta,\nonumber
\end{align*}
where, for any $i\in\nn$, $x_i$ denotes the centre of the dilated ball $B^{(i)}$ and
\begin{eqnarray}\label{e3.12.1}
\theta:=\lf(\frac{\ln b}{\ln \lambda_-}+s+1\r)\frac{\ln \lambda_-}{\ln b}>\frac{1}{\underline{p}}.
\end{eqnarray}
By this and the fact that, for any $x\in\rn$, $M_{N}(f)(x)\thicksim M_{N}^0(f)(x)$ (see \cite[Proposition 3.10]{b03}), we obtain
\begin{eqnarray}\label{e3.12}
M_{N}(\mathfrak{M}_{i})(x)\ls\lf\|\chi_{B^{(i)}}\r\|^{-1}_{L^{\vec{p}}}
\frac{\lf|{B^{(i)}}\r|^\theta}{[\rho(x-x_i)]^\theta}
\ls\lf\|\chi_{B^{(i)}}\r\|^{-1}_{L^{\vec{p}}}
\lf[M_{\mathrm{HL}}(\chi_{B^{(i)}})(x)\r]^\theta,
\end{eqnarray}
where, for any $i\in\nn$, $x_i$ denotes the centre of the dilated ball $B^{(i)}$ and $\theta$ is as in \eqref{e3.12.1}.
From this, Lemmas \ref{r2.1}, \ref{l3.5}, \eqref{e3.x} and \eqref{e3.3}, we deduce that
\begin{align}\label{e3.10x}
\mathrm{I}_{2}&\lesssim\lf\|\sum_{i\in\nn}\frac{|\lambda_i|}{\|\chi_{B^{(i)}}\|_{L^{\vec{p}}}}[M_{\mathrm{HL}}(\chi_{B^{(i)}})]^{\theta}\r\|_{L^{\vec{p}}}^{\underline{p}}\\
&\thicksim\lf\|\lf\{\sum_{i\in\nn}\frac{|\lambda_i|}{\|\chi_{B^{(i)}}\|_{L^{\vec{p}}}}[M_{\mathrm{HL}}(\chi_{B^{(i)}})]^{\theta}\r\}^{1/\theta}\r\|_{L^{\theta \vec{p}}}^{{\theta}{\underline{p}}}\nonumber\\
&\lesssim\lf\|\lf\{\sum_{i\in\nn}\frac{|\lambda_i|\chi_{B^{(i)}}}{\|\chi_{B^{(i)}}\|_{L^{\vec{p}}}}\r\}^{1/\theta}\r\|_{L^{\theta \vec{p}}}^{{\theta}{\underline{p}}}\nonumber
\thicksim\lf\|\sum_{i\in\nn}\frac{|\lambda_i|\chi_{B^{(i)}}}{\|\chi_{B^{(i)}}
\|_{L^{\vec{p}}}}\r\|_{L^{\vec{p}}}^{\underline{p}}\nonumber\\
&\lesssim\lf\|\lf\{\sum_{i\in\nn}\lf[\frac{|\lambda_i|\chi_{B^{(i)}}}
{\|\chi_{B^{(i)}}\|_{L^{\vec{p}}}}\r]^{\eta}\r\}^{1/\eta}\r\|_{L^{\vec{p}}}^{\underline{p}}\nonumber
\thicksim\|f\|_{H^{\vec{p},\,q,\,s,\,\ez}_{A,\,\mathrm{mol}}}^{\underline{p}},\nonumber
\end{align}
where $\eta\in(0,\,\min\{1,\,p_{-}\})$.
This, together with $\mathrm{I_1}$ and $\mathrm{I_2}$, shows that
\begin{align*}
\|f\|_{H_{A}^{\vec{p}}}\thicksim\|M_{N}(f)\|_{L^{\vec{p}}}\lesssim\|f\|_{H^{\vec{p},\,q,\,s,\,\ez}_{A,\,\mathrm{mol}}}.
\end{align*}
This implies that $f\in H^{\vec{p}}_{A}$ and hence
$H^{\vec{p},\,q,\,s,\,\varepsilon}_{A,\,\mathrm{mol}}\subset H^{\vec{p}}_{A}$. This finishes the proof of Theorem \ref{t2.6}.
\end{proof}

\section{Application}\label{s5}
\hskip\parindent
In this section, as an application, we obtain a boundedness criterion for some linear
operators from  $H^{\vec{p}}_{A}$ to itself.
Particularly, when $A:=2\rm I_{n\times n}$, this result is still new.

\begin{theorem}\label{t4.2}
Let $\vec{p}\in (0,\,\infty)^n$, $q\in(1,\,\infty]\cap(p_{+},\,\infty)$ with $p_+$ as in \eqref{e2.5}, and $s$ be as in \eqref{e2.10}. Suppose that $T$ is a bounded linear operator on $L^{r}$, for any $r\in(1,\,\infty]$. If for any $(\vec{p},\,q,\,s)$-atom $a$, supported on dilated ball $B^{(i_0)}\in\mathfrak{B}$, as in Definition \ref{d3.1}, $T(a)$ is a harmless constant multiple of a $(\vec{p},\,q,\,s,\,\varepsilon)$-molecule, associated with dilated ball $B^{(i_0)}\in\mathfrak{B}$, where $\varepsilon>(s+1)\log_b^{\lambda_+}$,
then $T$ extends uniquely to a bounded linear operator on $H^{\vec{p}}_{A}$. Moreover, there exists a positive constant $C$ such that, for all
$f\in H^{\vec{p}}_{A}$,
\begin{eqnarray}\label{e4.4x}
\|T(f)\|_{H^{\vec{p}}_{A}}\leq C\, \|f\|_{H^{\vec{p}}_{A}}.
\end{eqnarray}
\end{theorem}

To prove Theorem \ref{t4.2}, we need some definitions and technical lemmas.

\begin{definition}\label{d5.2}\rm{\cite{hlyy20}}
Let $\vec{p}\in (0,\,\infty)^n$, $q\in(1,\,\infty]$ and
$s\in[\lfloor(1/{p_-}-1) {\ln b/\ln \lambda_-}\rfloor,\,\infty)\cap\zz_+$ with $p_-$ as in \eqref{e2.5}.
 The {\it anisotropic mixed-norm finite atomic Hardy space}
$H^{\vec{p},\,q,\,s}_{A,\,\mathrm{fin}}$
is defined to be the set of all distributions $f\in \cs'$ satisfying that there
exist $I\in\nn$, $\{\lambda_i\}_{i\in[1,\,I]\cap\nn}\subset\ccc$ and a sequence of
$(\vec{p},\,q,\,s)$-atoms, $\{a_i\}_{i\in[1,\,I]\cap\nn}$, supported, respectively,
on $\{{B^{(i)}}\}_{i\in[1,\,I]\cap\nn}\subset\mathfrak{B}$ such that
\begin{align*}
f=\sum_{i=1}^I \lz_{i}a_i \ \ \mathrm{in\ } \ \mathcal{S}'.
\end{align*}
Moreover, for any $f\in H^{\vec{p},\,q,\,s}_{A,\,\rm{fin}}$, let
$$\|f\|_{H^{\vec{p},\,q,\,s}_{A,\,\rm{fin}}}
:=\inf \lf\|\lf\{\sum_{i=1} ^I \lf[\frac{|\lambda_i|\chi_{{B^{(i)}}}}{\|\chi_{{B^{(i)}}}\|
_{L^{\vec{p}}}}\r]^{\eta}\r\}^{1/\eta}\r\|
_{L^{\vec{p}}},$$
where $\eta\in(0,\,\min\{1,\,p_{-}\})$ and the infimum is taken over all the decompositions of $f$ as above.
\end{definition}

\begin{lemma}\label{l4.0y}\rm{\cite[Theorem 5.3]{hlyy20}}
Let $\vec{p}\in (0,\,\infty)^n$, $A$ be a dilation and
$s\in[\lfloor(1/{p_-}-1) {\ln b/\ln \lambda_-}\rfloor,\,\infty)\cap\zz_+$ with $p_-$ as in \eqref{e2.5}
\begin{enumerate}
\item[\rm{(i)}]If $q\in(\max\{p_+,\,1\},\,\infty)$ with $p_+$ as in \eqref{e2.5}, then $\|\cdot\|_{H^{\vec{p},\,q,\,s}_{A,\,\rm{fin}}}$ and $\|\cdot\|_{H^{\vec{p}}_{A}}$ are equivalent quasi-norms on $H^{\vec{p}}_{A}$.
\item[\rm{(ii)}]$\|\cdot\|_{H^{\vec{p},\,\infty,\,s}_{A,\,\rm{fin}}}$ and $\|\cdot\|_{H^{\vec{p}}_{A}}$ are equivalent quasi-norms on $H^{\vec{p},\,\infty,\,s}_{A,\,\rm{fin}}\cap C$,
    where $C$ denotes the set of all continuous functions on $\rn$.
\end{enumerate}
\end{lemma}
\begin{lemma}\label{l4.1yx}\rm{\cite[Theorem 4.7]{hlyy20}}
 Let $\vec{p}\in (0,\,\infty)^n$, $q\in(1,\,\infty]\cap(p_{+},\,\infty)$ with $p_+$ as in \eqref{e2.5}, $r\in(1,\,\infty]$ and $s\in[\lfloor(1/{p_-}-1) {\ln b/\ln \lambda_-}\rfloor,\,\infty)\cap\zz_+$ with $p_-$ as in \eqref{e2.5}. Then, for any $f\in H^{\vec{p}}_{A}\cap L^r$, there exist $\{\lambda_i\}_{{i\in\nn}}\subset\mathbb{C}$, dilated balls $\{x_i+B_{\ell_i}\}_{i\in\nn}\subset\mathfrak{B}$ and
 $(\vec{p},\,q,\,s)$-atoms $\{a_i\}_{{i\in\nn}}$ such that
 $$f=\sum_{i\in\mathbb N}\lambda_ia_i\ \ {\rm in}\ \ L^q,$$
 where the series also converge almost everywhere.
\end{lemma}


\begin{lemma}\label{l4.2xx}\rm{\cite{hlyy20}}
 Let $\vec{p}\in (0,\,\infty)^n$ and $N\in\mathbb{N}\cap[\lfloor(\frac{1}{\min\{1,\,p_-\}}-1)\frac{\ln b}{\ln \lambda_-}\rfloor+2,\,\infty)$ with $p_-$ as in \eqref{e2.5}. Then $H^{\vec{p}}_{A}$ is complete.
\end{lemma}

Now, we show Theorem \ref{t4.2} by borrowing some ideas from the proof of \cite[Theorem 8.4]{hlyy20}.
\begin{proof}[Proof of Theorem \ref{t4.2}]
Let $\vec{p}\in (0,\,\infty)^n$, $q\in(1,\,\infty]\cap(p_{+},\,\infty)$ with $p_+$ as in \eqref{e2.5}, and
$s\in[\lfloor(1/{p_-}-1) {\ln b/\ln \lambda_-}\rfloor,\,\infty)\cap\zz_+$ with $p_-$ as in \eqref{e2.5}.
Firstly, we show that \eqref{e4.4x} holds true for any $f\in H^{\vec{p},\,q,\,s}_{A,\,\mathrm{fin}}$. For any $f\in H^{\vec{p},\,q,\,s}_{A,\,\mathrm{fin}}$, by Definition \ref{d5.2}, we know that $f\in H^{\vec{p}}_{A}\cap L^q$.
From Lemma \ref{l4.1yx},
we know that there exist $\{\lambda_i\}_{i\in\nn}\subset\ccc$ and a sequence of
$(\vec{p},\,q,\,s)$-atoms, $\{a_i\}_{i\in\nn}$, supported, respectively,
on $\{{B^{(i)}}\}_{i\in\nn}\subset\mathfrak{B}$, where $B^{(i)}:=x_i+B_{\ell_i}$ with $x_i\in\rn$ and $\ell_i\in\zz$, such that
\begin{align*}
f=\sum_{i\in\nn}\lz_{i}a_i \ \ \mathrm{in\ } \ L^q,
\end{align*}
and
\begin{eqnarray}\label{e4.21}
 \lf\|\lf\{\sum_{i\in\nn} \lf[\frac{|\lambda_i|\chi_{{B^{(i)}}}}{\|\chi_{{B^{(i)}}}\|_{L^{\vec{p}}}}\r]
 ^{\eta}\r\}^{1/\eta}\r\|
_{L^{\vec{p}}}
\lesssim\|f\|_{H^{\vec{p}}_{A}}.
\end{eqnarray}
where $\eta\in(0,\,\min\{1,\,p_{-}\})$.
From this and the linear operator $T$ is bounded on $L^q$, we conclude that, for any $f\in H^{\vec{p},\,q,\,s}_{A,\,\mathrm{fin}}$,
$T(f)=\sum_{i\in\nn}\lz_{i}T(a_i)$ in $L^q$ and hence in $\cs'$.
Therefore, for all $N\in\mathbb{N}\cap[\lfloor({1/\underline{p}}-1)\ln b/\ln\lambda_{-}\rfloor+2,\,\infty)$,
\begin{align}\label{e4.52x}
&\lf\|M_{N}(T(f))\r\|_{L^{\vec{p}}}^{\underline{p}}
\leq\lf\|\sum_{i\in \nn}|\lambda_{i}|M_{N}(T(a_{i}))\r\|_{L^{\vec{p}}}^{\underline{p}}\\\nonumber
\leq&\lf\|\sum_{i\in \nn}|\lambda_{i}|M_{N}(T(a_{i}))\chi_{A^{2\sigma}B^{(i)}}\r\|_{L^{\vec{p}}}^{\underline{p}}+\lf\|\sum_{i\in \nn}|\lambda_{i}|M_{N}(T(a_{i}))\chi_{({A^{2\sigma}B^{(i)}})^{\complement}}\r\|_{L^{\vec{p}}}^{\underline{p}}\\\nonumber
\lesssim&\lf\|\lf\{\sum_{i\in \nn}\lf[|\lambda_{i}|M_{N}(T(a_{i}))\chi_{A^{2\sigma}B^{(i)}}\r]^
{\eta}\r\}^{1/\eta}\r\|_{L^{\vec{p}}}^{\underline{p}}+\lf\|\sum_{i\in \nn}|\lambda_{i}|M_{N}(T(a_{i}))\chi_{({A^{2\sigma}B^{(i)}})^{\complement}}\r\|_{L^{\vec{p}}}^{\underline{p}}\\\nonumber
=&:\mathrm{K_{1}}+\mathrm{K_{2}},\nonumber
\end{align}
where $A^{2\sigma}B^{(i)}$ is the $A^{2\sigma}$ concentric expanse on $B^{(i)}$, that is, $A^{2\sigma}B^{(i)}:=x_i+A^{2\sigma}B_{\ell_i}$,
 and $\underline{p}$ as in \eqref{e2.5.1}.

For ${\mathrm{K_1}}$, from the boundedness of $M_{N}$ and $T$ on $L^q$, and the size condition of $a_i$,
we know that
$$\lf\|M_{N}(T(a_i))\chi_{A^{2\sigma}B^{(i)}}\r\|_{L^q}
\lesssim\lf\|a_i\chi_{A^{2\sigma}B^{(i)}}\r\|_{L^q}\lesssim\frac{|B^{(i)}|^{1/q}}{\|\chi_{B^{(i)}}\|_{L^{\vec{p}}}}.$$
From this, Lemma \ref{l4.1x} and \eqref{e4.21}, we further deduce that
\begin{align*}
\mathrm{K_1}&\lesssim\lf\|\lf\{\sum_{i\in\nn} \lf[\frac{|\lambda_i|\chi_{{B^{(i)}}}}{\|\chi_{{B^{(i)}}}\|_{L^{\vec{p}}}}\r]
^{\eta}\r\}^{1/\eta}\r\|
_{L^{\vec{p}}}^{\underline{p}}
\lesssim\|f\|_{H^{\vec{p}}_{A}}^{\underline{p}}.
\end{align*}

To deal with ${\mathrm{K_2}}$, for any $i\in\nn$ and $x\in ({A^{2\sigma} B^{(i)}})^\complement$, by the condition of Theorem \ref{t4.2},  we see that, for any $(\vec{p},\,q,\,s)$-atom $a_i$ supported on a ball $B^{(i)}$,  $T(a_i)$ is a harmless constant multiple
of a $(\vec{p},\,q,\,s,\,\varepsilon)$-molecule associated with $B^{(i)}$, where
$\varepsilon>(s+1)\log_b^{\lambda_+}$. From this and \eqref{e3.12}, we know that
\begin{eqnarray}\label{e4.13}
M_N(T(a_i))(x)
\ls\lf\|\chi_{B^{(i)}}\r\|^{-1}_{L^{\vec{p}}}
\lf[M_{\mathrm{HL}}(\chi_{B^{(i)}})(x)\r]^\theta,
\end{eqnarray}
where
$\theta$  is as in \eqref{e3.12.1}.
By \eqref{e4.13} and an argument same as that used in the proof of \eqref{e3.10x}, we obtain
\begin{align*}
\mathrm{K_2}&\lesssim\lf\|\lf\{\sum_{i\in\nn} \lf[\frac{|\lambda_i|\chi_{{B^{(i)}}}}{\|\chi_{{B^{(i)}}}\|_{L^{\vec{p}}}}\r]
^{\eta}\r\}^{1/\eta}\r\|
_{L^{\vec{p}}}^{\underline{p}}
\lesssim\|f\|_{H^{\vec{p}}_{A}}^{\underline{p}}.
\end{align*}
Combining \eqref{e4.52x} and the estimates of $\mathrm{K_1}$ and $\mathrm{K_2}$, we further conclude that, for any $f\in H^{\vec{p},\,q,\,s}_{A,\,\mathrm{fin}}$,
$$\|T(f)\|_{H^{\vec{p}}_A}
\lesssim \|f\|_{H^{\vec{p}}_{A}}.$$

Next, we prove that \eqref{e4.4x} also holds true for any $f\in H^{\vec{p}}_{A}$.
Let $f\in H^{\vec{p}}_{A}$. By Lemma \ref{l4.0y} and the obvious density of $H^{\vec{p},\,q,\,s}_{A,\,\mathrm{fin}}$ in $H^{\vec{p}}_{A}$, we know that there exists a sequence
$\{f_j\}_{j\in {\mathbb{Z}_+}}\subset H^{\vec{p},\,q,\,s}_{A,\,\mathrm{fin}}$, such that $f_j \rightarrow f$ as $j\rightarrow\infty$ in $H^{\vec{p}}_A$.
Therefore, $\{f_j\}_{j\in {\mathbb{Z}_+}}$ is a Cauchy sequence in $H^{\vec{p}}_A$. By this, we see that, for any $j$, $k\in {\mathbb{Z}_+}$,
$$\|T(f_j)-T(f_k)\|_{H_A^{\vec{p}}}=\|T(f_j-f_k)\|_{H_A^{\vec{p}}}\lesssim\|f_j-f_k\|_{H^{\vec{p}}_A}.$$
Notice that $\{T(f_j)\}_{j\in\mathbb{Z}_+}$ is also a Cauchy sequence in $H_A^{\vec{p}}$. Applying Lemma \ref{l4.2xx}, we conclude that there exists a $g \in H_A^{\vec{p}}$ such that $T(f_j)\rightarrow g$ as $j\rightarrow\infty$ in $H_A^{\vec{p}}$.
Let $T(f):=g$. Then, $T(f)$ is well defined. In fact, for any other sequence $\{h_j\}_{j\in\mathbb{Z}_+}\subset H^{\vec{p},\,q,\,s}_{A,\,\mathrm{fin}}$ satisfying $h_j\rightarrow f$ as $j\rightarrow\infty$ in $H^{\vec{p}}_A$, by Lemma  \ref{r2.1}, we have
\begin{align*}
\|T(h_j)-T(f)\|_{{H^{\vec{p}}_A}}^{\underline{p}}&\leq\|T(h_j)-T(f_j)\|_{{H^{\vec{p}}_A}}^{\underline{p}}
+\|T(f_j)-g\|_{{H^{\vec{p}}_A}}^{\underline{p}}.\\
&\lesssim\|h_j-f_j\|_{H^{\vec{p}}_A}^{\underline{p}}+\|T(f_j)-g\|_{{H^{\vec{p}}_A}}^{\underline{p}}\\
&\lesssim\|h_j-f\|_{H^{\vec{p}}_A}^{\underline{p}}+\|f-f_j\|_{H^{\vec{p}}_A}^{\underline{p}}+\|T(f_j)-g\|_{{H^{\vec{p}}_A}}^{\underline{p}}\rightarrow 0 \,\,\mathrm{as}\,\, j\rightarrow 0,
\end{align*}
which is wished.

From this, we see that, for any $f\in H^{\vec{p}}_A$,
$$\|T(f)\|_{{H^{\vec{p}}_A}}=\|g\|_{{H^{\vec{p}}_A}}=\lim_{j\rightarrow\infty}\|T(f_j)\|_{{H^{\vec{p}}_A}}\lesssim\lim_{j\rightarrow\infty}\|f_j\|_{H^{\vec{p}}_A} \thicksim\|f\|_{H^{\vec{p}}_A},$$
which implies that \eqref{e4.4x} also holds true for any $f\in H^{\vec{p}}_{A}$ and hence completes the proof of Theorem \ref{t4.2}.
\end{proof}

\textbf{Acknowledgements.} The authors would like to express their
deep thanks to the referees for their very careful reading and useful
comments which do improve the presentation of this article.

\bigskip

\noindent
\medskip
\noindent  Aiting Wang
\medskip

\noindent
School of Mathematics and Statistics\\
Qinghai Nationalities University\\
 Xining
810000, Qinghai, China\\

\noindent

Wenhua Wang

\medskip

\noindent
School of Mathematics and Statistics\\
Wuhan University\\
 Wuhan
430072, Hubei, China

\smallskip

\noindent{E-mail }:\\
\texttt{wangwhmath@163.com} (Wenhua Wang)\\
\texttt{atwangmath@126.com} (Aiting Wang) \\

\bigskip \medskip

\end{document}